\documentclass[11pt]{amsart}
\usepackage{amssymb}
\usepackage[mathcal]{euscript}
\usepackage{mathrsfs}
\usepackage[english]{babel}
\usepackage[T2A]{fontenc}
\usepackage[cp1251]{inputenc}
\theoremstyle{plain}
\newtheorem*{proposition*}{Твердження}
\newtheorem{theorem}{Theorem}

\newtheorem{lemma}{Lemma}
\theoremstyle{definition*}
\newtheorem*{definition*}{Definition}
\newtheorem{ex}{Example}
\newtheorem*{corollary*}{Наслідок}
\newtheorem{corollary}{Corollary}
\theoremstyle{remark}
\newtheorem*{remark*}{Remark}

 \uccode`ґ=`Ґ \uccode`Ґ=`Ґ \lccode`Ґ=`ґ \lccode`ґ=`ґ \uccode`є=`Є
\uccode`Є=`Є \lccode`Є=`є \lccode`є=`є \uccode`і=`І \uccode`І=`І
\lccode`І=`і \lccode`і=`і \uccode`ї=`Ї \uccode`Ї=`Ї \lccode`Ї=`ї
\lccode`ї=`ї

\begin{document}

\title[]{On singularity of distribution of random variables with independent symbols of Oppenheim expansions}
\author[L. Sinelnyk, G. Torbin] {Liliia Sinelnyk,  Grygoriy Torbin}

\begin{abstract}
The paper is devoted to restricted Oppenheim expansion  of real numbers ($ROE$),which includes as partial cases already known  Engel, Silvester and L\"uroth expansions. We find conditions under which for almost all (with respect to Lebesgue measure) real numbers from the unit interval their $ROE$-expansion contain arbitrary digit $i$ only finitely many times. Main results of the paper states the singularity (w.r.t. Lebesgue measure) of the distribution of random variable with i.i.d increments of symbols of restricted Oppenheim expansion. General non-i.i.d. case are also studied and sufficient conditions for the singularity of the corresponding probability distributions are found.
  \end{abstract}

\maketitle
\medskip


\section{Introduction}
It is well known that singularly continuous probability measures were studied during almost all XX century and there are a lot of open problems related to them. The fractal and multifractal  approaches to the study of  such measures are known to be extremely useful (see, e.g., \cite{AKPT2011, BPT, TorbinUMJ2005} and references therein). Study of fractal properties of different families of singularly continuous probability measures (see, e.g., \cite{ AKPT2011, GT2012, Lupain NZ2012, Lupain MSTA, Lupain NZ2014, Lupain Torbin NZ2014, NikTor2012TIMS, Torbin 2002} and references therein) can be used to solve non-trivial problems in metric number theory (\cite{AKNT2, AGIT,   APT3, APT4, AKNT1 and 3, NikTor NZ 2012, ILT2}), in the theory of dynamical systems and DP-transformations and in fractal analysis (\cite{AKPT, apt2004, AGPT,  GNT2016, GNT NZ2014(1), GNT NZ2014(2), INT,  Torbin SP 2007, TP}).

On the other hand for many families of probability measures the problem "singularity vs absolute continuity" are extremely complicated even for the so-called probability distributions of the Jessen-Wintner type, i.e., distributions of random variables which are sums of almost surely convergent series of independent discretely distributed random variables. Infinite Bernoulli convolutions form one important subclass of such measures  (see, e.g.,  \cite{AlbeverioTorbin2007, PSS, ST3,  ST2, ST1,  Solomyak} and references therein). Another wide family of probability measures where the problem "singularity vs absolute continuity" are still open consists of probability distributions of the following form:

$$
\xi = \Delta_{\xi_1 \xi_2 ... \xi_n ... }^{F},
$$
where $\xi_n$ are independent symbols of some generalize $F$-expansion over some alphabet $A$. Random variables with independent symbols of $s-$adic expansions, continued fraction expansions, L\"uroth expansion, Silvester and Engel expansions are among them.  This paper is devoted to the development of probabilistic theory of Oppenheim expansions of real numbers which contains many important expansions as rather special cases. Main result of the paper shows that in this family of probability measures the singularity are generic.

\section{On metric theory of restricted Oppenheim expansion}

It is known (\cite{Galambos}) that any real number $x \in (0,1)$ can be represented in the form of the Oppenheim expansion
$$x\sim \frac{1}{d_1}+\frac{a_1}{b_1}\frac{1}{d_2}+\ldots+\frac{a_1 a_2\cdot\ldots\cdot a_n}{b_1 b_2\cdot\ldots\cdot b_n}\frac{1}{d_{n+1}}+\ldots \eqno(1)$$
where $a_n=a_n(d_1, \ldots, d_n)$, $b_n=b_n(d_1, \ldots, d_n)$ are positive integers and the denominators $d_n$ are determined by the algorithm:
$$x=x_1;$$
$$d_n=\left[\frac{1}{x_n}\right]+1;$$
$$x_n=\frac{1}{d_n}+\frac{a_n}{b_n}x_{n+1}; \eqno(2)$$
and satisfy inequalities:
$$d_{n+1}>\frac{a_n}{b_n}d_n(d_n-1).$$

 A sufficient condition for a series on the right-hand side in (1) to be the expansion of its sum by the algorithm (2) is:
$$d_{n+1}\geq\frac{a_n}{b_n}d_n(d_n-1)+1.$$

We call the expansion (1) (obtained by the algorithm (2)) the restricted Oppenheim expansion (ROE) of $x$ if $a_n$ and $b_n$ depend only on the last denominator $d_n$ and if the function
$$h_n(j):=\frac{a_n(j)}{b_n(j)}j(j-1)\eqno(3)$$
is integer valued.

Let us consider some examples of restricted Oppenheim expansions.

\begin{ex}
 Let $a_n=1$, $b_n=d_n$, ($n=1,2,\ldots$).Then the expansion (1) obtained by the algorithm (2) is the well known Engel expansion of $x$:
$$x=\frac{1}{d_1}+\frac{1}{d_1d_2}+\ldots+\frac{1}{d_1d_2\ldots d_n}+\ldots,$$
where $d_{n+1}\geq d_n.$
\end{ex}

\begin{ex}
Let $a_n=b_n=1$ (or $a_n=b_n=const$) ($n=1,2,\ldots$). Then the expansion (1) obtained by the algorithm (2) is the well known Silvester expansion of $x$:
$$x=\frac{1}{d_1}+\frac{1}{d_2}+\ldots+\frac{1}{d_n}+\ldots,$$
where $d_{n+1}\geq d_n(d_n-1)+1.$

\end{ex}

\begin{ex}
Let $a_n=1,$ $b_n=d_n(d_n-1)$. In this case we obtain L\"uroth series for a number $x:$
$$x=\frac{1}{d_1}+\frac{1}{d_1(d_1-1)d_2}+\ldots+\frac{1}{d_1(d_1-1)\ldots d_n(d_n-1)d_{n+1}}+\ldots,$$
where $d_{n+1}\geq 2.$
\end{ex}

It is known \cite{Galambos} that metric, dimensional and probabilistic theories of  Oppenheim series are underdeveloped (in fact, as evidenced by its recent work and dissertation (\cite{Zhyhareva}, \cite{Hetman}, \cite{PZ}) even such partial cases of Oppenheim  expansions as Luroth series, Engel and Silvester series generate a number of challenges metric and probabilistic number theory). The main purpose of this article is to develop some general methods of the metric theory of numbers and Oppenheim expansions, show their effectiveness in the study of Lebesgue structures of distributions of random variables with independent symbols of Oppenheim expansions.

Choose the probability space $(\Omega, \mathcal{A}, P),$ as $\Omega=(0,1)$, $\mathcal{A}$ the set of Lebesgue measurable subsets of $(0,1)$ and Lebesgue measure as $P$.

Let $\triangle_{j_1 j_2 \ldots j_n}^{ROE}:=\{x: d_1(x)=j_1, d_2(x)=j_2,\ldots, d_n(x)=j_n\}$ be the cylinder of rank $n$ with base $(j_1, j_2,\ldots,j_n).$
\begin{lemma}(\cite{Galambos})
      $$P(d_1=j_1, \ldots, d_n=j_n)=\frac{a_1a_2\cdot\ldots\cdot a_{n-1}}{b_1b_2\cdot\ldots\cdot b_{n-1}}\frac{1}{j_n(j_n-1)}$$
where $a_i=a_i(j_i)$, $b_i=b_i(j_i)$ $(i=1,2,\ldots,n-1)$ .
\end{lemma}
\begin{proof}

At first calculate $P(d_1=j_1).$
$$P(d_1=j_1)=P\{x: d_1(x)=j_1\}=\vert\triangle_{j_1}^{ROE}\vert=\left|\left(\frac{1}{j_1}, \frac{1}{j_1-1}\right)\right|=\frac{1}{j_1(j_1-1)}.$$
$$P(d_1=j_1, d_2=j_2)=P\{x:x\in \triangle_{
j_1j_2}^{ROE}\}=P\{x:x=\frac{1}{j_1}+\frac{a_1}{b_1}\frac{1}{j_2}+\ldots\};$$
$$inf\triangle_{j_1j_2}^{ROE}=\frac{1}{j_1}+\frac{a_1}{b_1}\frac{1}{j_2};$$
$$sup\triangle_{j_1j_2}^{ROE}=inf\triangle_{j_1(j_2-1)}^{ROE}=\frac{1}{j_1}+\frac{a_1}{b_1}\frac{1}{j_2-1};$$
$$\triangle_{j_1j_2}^{ROE}=\left(\frac{1}{j_1}+\frac{a_1}{b_1}\frac{1}{j_2}, \frac{1}{j_1}+\frac{a_1}{b_1}\frac{1}{j_2-1}\right).$$
As follows
$$P(d_1=j_1, d_2=j_2)=|\triangle_{j_1j_2}^{ROE}|=\frac{a_1}{b_1}\frac{1}{j_2(j_2-1)}.$$

Similarly
$$P(d_1=j_1, d_2=j_2, \ldots, d_n=j_n)=P\{x\in\triangle_{j_1j_2\ldots j_n}^{ROE}\};$$
$$inf\triangle_{j_1j_2\ldots j_n}^{ROE}=\frac{1}{j_1}+\frac{a_1}{b_1}\frac{1}{j_2}+\ldots+\frac{a_1a_2\ldots a_n}{b_1b_2\ldots b_n}\frac{1}{j_n};$$
$$sup\triangle_{j_1j_2\ldots j_n}^{ROE}=inf\triangle_{j_1j_2\ldots j_{n-1} (j_n-1)}^{ROE}=\frac{1}{j_1}+\frac{a_1}{b_1}\frac{1}{j_2}+\ldots+\frac{a_1a_2\ldots a_{n-1}}{b_1b_2\ldots b_{n-1}}\frac{1}{j_n-1};$$
$$|\triangle_{j_1j_2\ldots j_n}^{ROE}|=\frac{a_1a_2\ldots a_{n-1}}{b_1b_2\ldots b_{n-1}}\frac{1}{j_n(j_n-1)}.$$
\end{proof}

\begin{theorem}  (\cite{Galambos})
The sequence $d_n (n=1,2,\ldots)$ forms a Markov chain
$$P(d_1=j)=\frac{1}{j(j-1)};$$
$$P(d_n=k|d_{n-1}=j)=\frac{h_{n-1}(j)}{k(k-1)}, \quad k>h_{n-1}(j);$$
and 0 otherwise.
\end{theorem}
\begin{proof}
Since $\{d_1=j_1,\ldots,d_{n-1}=j\}>0,$ then
$$P(d_n=k|d_1=j_1,\ldots,d_{n-1}=j)=\frac{P(d_{n-1}=j,\ldots, d_n=k)}{P(d_{n-1}=j,\ldots,d_1=j_1)}.$$
Thus, by Lemma, this ratio is equal to:
$$\frac{a_1a_2\ldots a_{n-1}}{b_1b_2\ldots b_{n-1}}\frac{1}{k(k-1)}:\frac{a_1a_2\ldots a_{n-2}}{b_1b_2\ldots b_{n-2}}\frac{1}{j(j-1)}=\frac{a_{n-1}(j)}{b_{n-1}(j)}\frac{j(j-1)}{k(k-1)}=\frac{h_{n-1}(j)}{k(k-1)}, \quad k>h_{n-1}(j)$$
\end{proof}

 Therefore,  we get the following  properties of cylinders:
 %\begin{itemize}

 1)  $\triangle_{j_1 j_2 \ldots j_{n-1}}^{ROE}=\bigcup\limits_{i=1}^\infty \triangle_{j_1 j_2 \ldots j_{n-1}i}^{ROE}.$

\bigskip

 2) $sup \triangle_{j_1 j_2 \ldots j_n}^{ROE}=inf \triangle_{j_1 j_2 \ldots j_{n-1}j_n+1}^{ROE}.$

\bigskip
 3) $inf\triangle_{j_1j_2\ldots j_n}^{ROE}=\frac{1}{j_1}+\frac{a_1}{b_1}\frac{1}{j_2}+\ldots+\frac{a_1a_2\ldots a_n}{b_1b_2\ldots b_n}\frac{1}{j_n},$

\bigskip
 $~~~sup\triangle_{j_1j_2\ldots j_n}^{ROE}=\frac{1}{j_1}+\frac{a_1}{b_1}\frac{1}{j_2}+\ldots+\frac{a_1a_2\ldots a_{n-1}}{b_1b_2\ldots b_{n-1}}\frac{1}{j_n-1}.$

\bigskip
 4)  $|\triangle_{j_1j_2\ldots j_n}^{ROE}|=\frac{a_1a_2\ldots a_{n-1}}{b_1b_2\ldots b_{n-1}}\frac{1}{j_n(j_n-1)}.$
   \bigskip

If first $k$ symbols of ROE are fixed,  then $(k+1)-th$ symbol of ROE can not take values $2,3,\ldots,\frac{a_k}{b_k}d_k(d_k-1),$ $\forall k \in \mathbb{N}.$

Each of the cylinders of ROE  can be uniquely rewritten in terms of the difference restricted Oppenheim expansion ($\overline{ROE}$):
$$\alpha_1=d_1-1;$$
$$\alpha_{k+1}=d_{k+1}-\frac{a_k}{b_k}d_k(d_k-1).$$

Then series (1) can be rewritten as follows:
$$x=\frac{1}{\alpha_1+1}+\frac{a_1}{b_1}\frac{1}{\frac{a_1}{b_1}d_1(d_1-1)+\alpha_2}+\frac{a_1a_2}{b_1b_2}\frac{1}{\frac{a_2}{b_2}d_2(d_2-1)+\alpha_3}+\ldots =: \Delta_{\alpha_1 \alpha_2 ... \alpha_n ...}^{\overline{ROE}}.$$

where $\alpha_k\in \{1, 2,3,\ldots\}.$

\begin{theorem}\label{theorem *}
If there exists a sequence $l_k$, such that $\forall x\in [0,1]:$
$$\frac{a_{k-1}}{b_{k-1}}\frac{b_k^2}{a_k^2}\frac{1}{d_{k-1}(d_{k-1}-1)}<l_k$$
and series
$$\sum_{k=1}^\infty l_k<+\infty,$$
then for any digit $i_0$ almost all (with respect to the Lebesgue measure) real numbers  $x\in [0,1]$ contain symbol $i_0$ only finitely many times in $\overline{ROE}.$
\end{theorem}
\begin{proof}
Let  $N_i(x)$ be a number of symbols  $"i"$ in $\overline{ROE}$ of number $x.$ Let us prove that Lebesgue measure of set $A_i=\{x:N_i(x)=\infty\}$ is equal to 0 for all  $i \in \mathbb{N}.$

Consider the set
$$\bar{\triangle}_i^k=\{x:x=\triangle^{\overline{ROE}}_{\alpha_1 \alpha_2 \ldots \alpha_{k-1} i \alpha_{k+1} \ldots}, \alpha_j \in \mathbb{N}, j\neq k\}.$$

From the definition of the set $\bar{\triangle}_i^k$ and properties of cylindrical sets it follows that

$$\bar{\triangle}_i^k=\bigcup_{\alpha_1=1}^\infty\ldots\bigcup_{\alpha_{k-1}=1}^\infty \triangle^{\overline{ROE}}_{\alpha_1\ldots \alpha_{k-1} i}$$

 Let us consider the following ratio:

$$\frac{|\triangle^{\overline{ROE}}_{\alpha_1\ldots \alpha _{k-1}i}|}{|\triangle^{\overline{ROE}}_{\alpha_1\ldots \alpha _{k-1}}|}=\frac{\left|\triangle^{ROE}_{d_1\ldots d_{k-1}\left(\frac{a_k}{b_k}d_{k-1}(d_{k-1}-1)+i\right)}\right|}{|\triangle^{ROE}_{d_1\ldots d_{k-1}}|}=$$
$$=\frac{a_1\ldots a_{k-1}}{b_1\ldots b_{k-1}}\cdot \frac{1}{{\left(\frac{a_k}{b_k}d_{k-1}(d_{k-1}-1)+i\right)}\left(\frac{a_k}{b_k}d_{k-1}(d_{k-1}-1)+i-1\right)}: \frac{a_1\ldots a_{k-2}}{b_1\ldots b_{k-2}}\cdot \frac{1}{d_{k-1}(d_{k-1}-1)}=$$
$$=\frac{a_{k-1}}{b_{k-1}}\cdot\frac{d_{k-1}(d_{k-1}-1)}{\frac{a_k}{b_k}d_{k-1}(d_{k-1}-1)+i}\cdot\frac{1}{\frac{a_k}{b_k}d_{k-1}(d_{k-1}-1)+i-1}\leq$$
$$\leq \frac{a_{k-1}}{b_{k-1}}\cdot \frac{d_{k-1}(d_{k-1}-1)}{\frac{a_k}{b_k}d_{k-1}(d_{k-1}-1)}\cdot \frac{1}{\frac{a_k}{b_k}d_{k-1}(d_{k-1}-1)}=$$
$$=\frac{a_{k-1}}{b_{k-1}}\cdot \frac{b_k^2}{a_k^2}\cdot \frac{1}{d_{k-1}(d_{k-1}-1)}<l_k$$
Then
$$\lambda(\bar{\triangle}_i^k)=\sum\limits_{\alpha_1=1}^\infty\ldots\sum\limits_{\alpha_k=1}^\infty|\triangle^{\overline{ROE}}_{\alpha_1\ldots \alpha_{k-1}i}|\leq  l_k.$$

It is clear, that the set $A_i$ is the upper limit of the sequence of sets $\{\bar{\triangle}_i^k\},$ i.e.,
$$A_i=\limsup_{k\to\infty}\bar{\triangle}_i^k=\bigcap_{m=1}^\infty\left(\bigcup_{k=m}^\infty\bar{\triangle}_i^k\right).$$

Since
$$\sum_{k=1}^\infty \lambda (\bar{\triangle}_i^k)\leq\sum_{k=1}^\infty \frac{a_{k-1}}{b_{k-1}}\frac{b_k^2}{a_k^2}\frac{1}{d_{k-1}(d_{k-1}-1)}\leq \sum_{k=1}^\infty l_k <+\infty,$$
from Borel-Cantelli Lemma it follows that
$$\lambda(A_i)=0, \quad \forall i \in N.$$

Therefore,
$$\lambda(\bar{A}_i)=1, \quad \forall i \in N.$$

Let
$$\bar{A}=\bigcap_{i=1}^\infty \bar{A}_i.$$
It is clear that $\lambda(\bar{A})=1,$ which proves the theorem.
\end{proof}

\begin{ex}
Consider the Silvester series:
$$d_1\in\{2,3,\ldots\},$$
$$d_{k+1}=d_k(d_k-1)+i, \quad i \in \{1,2,3,\ldots\}.$$
If $d_1=2,$ then $\min d_2=3.$
Therefore
$$d_{k+1}\geq d_k(d_k-1)+1\geq (d_{k-1}(d_{k-1}-1)+1)(d_{k-1}(d_{k-1}-1)+1)\geq (d_{k-1}(d_{k-1}-1))^2+1\geq$$
$$\geq((d_{k-2}(d_{k-2}-1)+1)(d_{k-2}(d_{k-2}-1)+1-1))^2+1\geq (d_{k-2}(d_{k-2}-1))^4\geq(d_{k-3}(d_{k-3}-1))^{2^3}\geq\ldots\geq$$
$$\geq (d_{k-(k-2)}(d_{k-(k-2)}-1))^{2^{k-2}}=(d_2(d_2-1))^{2^{k-2}}\geq3\cdot 2^{2^{k-2}}.$$

So for the Silvester series:
$$\frac{1}{d_{k-1}(d_{k-1}-1)}<\frac{1}{3\cdot 2^{2^{k-4}}\cdot (3\cdot 2^{2^{k-4}})}=:l_k.$$
It is clear that
$$\sum\limits_{k=1}^\infty l_k<\infty.$$

Therefore, for  $\lambda$-almost all $x\in[0,1]$ their Silvester series contain arbitrary digit $i$ only finitely many times.
\end{ex}

\begin{ex}
Consider the case  where $a_n=d_n,$ $b_n=1.$ Then
$$d_{n+1}\geq d_n\cdot d_n (d_n-1)+1\geq d_n^2 \geq(d_{n-1}^2)^2=$$
$$=d_{n-1}^4\geq d_{n-2}^8\geq d_{n-3}^{2^4}\geq\ldots\geq d_{n-(n-1)}^{2^n}=d_1^{2^n}\geq 2^{2^n}.$$

So for this case
$$\frac{1}{d_{k-1}(d_{k-1}-1)}<\frac{1}{2^{2^k}}=:l_k.$$

Then,
$$\sum\limits_{k=1}^\infty l_k<\infty.$$

So for $\lambda$-almost all $x\in[0,1]$ the expansion contains arbitrary digit $i$ only finitely many times.
\end{ex}

\section{On singularity of distribution of random variables with independent symbols of $\overline{ROE}$}

\begin{definition*}
A probability measure $\mu_\xi$ of a random variable $\xi$ is said to be singularly continuous (with respect to Lebesgue measure) if $\mu_\xi$ is a continuous probability measure and there exists a set $E$,  such that $\lambda(E)=0$ and $\mu_\xi(E)=1.$
\end{definition*}

Let $x=\Delta^{\overline{ROE}}_{\alpha_1(x)\alpha_2(x)...\alpha_n(x)...}$ be  $\overline{ROE}$ of real numbers, let $\xi_1, \xi_2,\ldots, \xi_k, \ldots$ be a sequence of independent random variables taking values $1,2,\ldots, n, \ldots$ with probabilities $p_{1k}, p_{2k},\ldots, p_{nk}, \ldots$ correspondingly, and let $$\xi=\Delta^{\overline{ROE}}_{\xi_1 \xi_2...\xi_n...}$$ be a random variables with independent  $\overline{ROE}$- symbols.
\begin{theorem}
Let assumptions of Theorem \ref{theorem *} hold. If there exists a digit $i_0$ such that  $\sum\limits_{k=1}^\infty p_{i_0 k}=+\infty,$ then
the probability measure $\mu_\xi$ is singular with respect to Lebesgue measure.
\end{theorem}

\begin{proof}

Consider sets $$B_n=\{x:\alpha_n(x)=i_0\}$$ and $$B=\{x:N_{i_0}(x)=+\infty\}.$$

It is clear, that $B=\varlimsup\limits_{n\to\infty}B_n.$

From the definition of $B_n$ it follows that $\mu_\xi(B_n)=p_{i_0n}.$

Since the random variables $\xi_1, \xi_2,\ldots,\xi_n,\ldots$ are independent, we conclude that
$$\mu_\xi\left(B_{k_1}\cap B_{k_2}\cap\ldots\cap B_{k_s}\right)=\mu_\xi\left(\left\{x:\alpha_{k_1}(x)=i_0, \alpha_{k_2}(x)=i_0,\ldots, \alpha_{k_s}(x)=i_0\right\}\right)=$$
$$=\mu_\xi\left(\left\{x: \alpha_{k_1}(x)=i_0\right\}\right)\cdot\mu_\xi\left(\left\{x: \alpha_{k_2}(x)=i_0\right\}\right)\cdot\ldots\cdot\mu_\xi\left(\left\{x: \alpha_{k_s}(x)=i_0\right\}\right)=$$
$$=p_{i_0k_1}\cdot p_{i_0 k_2}\cdot\ldots\cdot p_{i_0 k_s}.$$

So, events $B_1,B_2,\ldots,B_n,\ldots$ are independent  with respect to  measure $\mu_\xi.$

Since $\sum\limits_{k=1}^\infty p_{i_0 k}=+\infty$ and $\{B_n\}$ is a sequence of independent events from Borel-Cantelli Lemma it follows that $$\mu_\xi(B)=1.$$

Let $\lambda$ be Lebesgue measure. Events $B_1, B_2,\ldots,B_n,\ldots$ are generally speaking,  not independent w.r.t. Lebesgue measure. We estimate the Lebesgue measure of the set $B_n$:
$$\lambda(B_n)=\lambda\left(\left\{x:\alpha_n(x)=i_0\right\}\right)=\sum\limits_{\alpha_1(x)\in A}\sum\limits_{\alpha_2(x)\in A}\ldots\sum\limits_{\alpha_{n-1}\in A}\left|\Delta^{\overline{ROE}}_{\alpha_1(x) \alpha_2(x)\ldots \alpha_{n-1}(x) i_0}\right|=$$
$$\sum\limits_{\alpha_1(x)\in A}\sum\limits_{\alpha_2(x)\in A}\ldots\sum\limits_{\alpha_{n-1}\in A}\frac{\left|\Delta_{\alpha_1(x) \alpha_2(x)\ldots \alpha_{n-1}(x) i_0}^{\overline{ROE}}\right|}{\left|\Delta_{\alpha_1(x) \alpha_2(x)\ldots \alpha_{n-1}(x) }^{\overline{ROE}}\right|}\cdot \left|\Delta_{\alpha_1(x) \alpha_2(x)\ldots \alpha_{n-1}(x) }^{\overline{ROE}}\right|\leq$$
$$\leq l_n(i_0)\cdot\sum\limits_{\alpha_1(x)\in A}\sum\limits_{\alpha_2(x)\in A}\ldots\sum\limits_{\alpha_{n-1}\in A}\left|\Delta_{\alpha_1(x) \alpha_2(x)\ldots \alpha_{n-1}(x) }^{\overline{ROE}}\right|=l_n \cdot 1.$$

Therefore, $$\sum\limits_{n=1}^\infty \lambda (B_n)\leq\sum\limits_{n=1}^\infty l_n<+\infty.$$

So by Borel-Cantelli Lemma , $\lambda(B)=0,$ i.e. for $\lambda$-almost all $x\in[0,1]$ their $\overline{ROE}$ contains arbitrary digit $i$ only finitely many times.

Hence $\lambda(B)=0$, аnd $\mu_\xi(B)=1$. So, probability measure $\mu_\xi$ is singular with respect to Lebesgue measure
\end{proof}

\begin{theorem}
Let assumptions of Theorem \ref{theorem *} hold. If  $\xi_k$ are  independent and identically distributed random variables, then the probability measure $\mu_\xi$ is singular with respect to Lebesgue measure.
\end{theorem}

\begin{proof}
If $\xi_1, \xi_2, \ldots, \xi_n, \ldots$ are independent and identically distributed random variables, then the matrix $P=\Vert p_{ik}\Vert$ is of the following form
$$P=\left(
            \begin{array}{ccccc}
              p_1 & p_1 & \ldots & p_1 & \ldots \\
              p_2 &  p_2 & \ldots& p_2 & \ldots \\
              \ldots &  \ldots & \ldots & \ldots & \ldots \\
              p_k & p_k  & \ldots & p_k & \ldots \\
            \ldots & \ldots & \ldots & \ldots & \ldots \\
            \end{array}
          \right)$$

Since $\sum\limits_{i=1}^\infty p_i=1,$ it is clear that there exists a  number $i_0$ such, that: $p_{i_0}>0$. Therefore $$\sum\limits_{k=1}^\infty p_{i_0k}=+\infty.$$
\end{proof}

\begin{corollary}
Let $$x=\Delta^{\overline{S}}_{\alpha_1(x)\alpha_2(x)...\alpha_n(x)...}$$ be the difference version of Silvester expansion $(\overline{S}$-expansion) and let $$\xi=\Delta^{\overline{S}}_{\xi_1 \xi_2...\xi_n...}$$ be the random variable with independent symbols of $\overline{S}$-expansion.

If there exists a digit $i_0$ such that  $\sum\limits_{k=1}^\infty p_{i_0 k}=+\infty,$ then the probability measure $\mu_\xi$ is singular with respect to Lebesgue measure.

In particular, the distribution of the random variable with independent identically distributed symbols of $\overline{S}$-expansion is singular w.r.t. Lebesgue measure.
\end{corollary}


\begin{thebibliography}{999}

\bibitem{AKNT2}~Albeverio S., Kondratiev Yu., Nikiforov R., Torbin G. On fractal properties of non-normal numbers with respect to R\'{e}nyi $f$-expansions generated by piecewise linear functions, \emph{Bull. Sci. Math.}, \textbf{138} (2014), no.~3,  440~--~455.


\bibitem{AGIT}~Albeverio S., Garko I., Ibragim M., Torbin G. Non-normal numbers: Full Hausdorff dimensionality vs zero dimensionality.
\textit{Bulletin des Sciences Mathйmatiques}, \textbf{141} (2017), No. 2, 1-19.


\bibitem{AKPT} Albeverio S.,  Koshmanenko V., Pratsiovytyi M., Torbin G. Spectral properties of image measures under the infinite conflict interactions,  \emph{Positivity},  \textbf{10} (2006),  no.1,  39-49.

\bibitem{AKPT2011} Albeverio S.,  Koshmanenko V., Pratsiovytyi M., Torbin G. On fine structure of singularly continuous  probability measures and random variables with independent  $\widetilde{Q}$- symbols, \textit{Meth. of Func. An. Top.}, \textbf{17} (2011), no.2,  P.97--111.

\bibitem{AT2} Albeverio S., Torbin G. Fractal properties of singularly continuous probability distributions with independent $Q^{*} $ - digits. \emph{Bull.
Sci. Math.},\textbf{129} (2005), No 4, 356--367.


\bibitem{AlbeverioTorbin2007}  Albeverio S., Torbin G. On fine fractal properties of generalized infinite Bernoulli
convolutions, \textit{ Bull. Sci. Math.}, \textbf{132}(2008), no. 8.  711-727.

\bibitem{APT3} Albeverio S., Pratsiovytyi M., Torbin G. Topological and fractal
properties of subsets of real numbers which are not normal.
\textit{Bull.Sci.Math.}, \textbf{129} (2005), no. 8, 615--630.

\bibitem{APT4} Albeverio S., Pratsiovytyi M., Torbin G.  Singular probability
distributions and fractal properties of sets of real numbers defined
by the asymptotic frequencies of their s-adic digits,
\textit{Ukr.Math.J.}, \textbf{57} (2005), no.9, 1361-1370.

\bibitem{apt2004} Albeverio S., Pratsiovytyi M., Torbin G. Fractal probability distributions and transformations preserving the Hausdorff-Besicovitch dimension. \emph{Ergodic Theory and Dynamical Systems}. --200,. Vol.~24, no.~1. -- P.~1--16.

\bibitem{AGPT} Albeverio S.,  Goncharenko Ya., Pratsiovytyi M., Torbin G. Convolutions of distributions of random variables with independent binary digits.
\textit{Random Operators and Stochastic Equations}, \textbf{15}(1), 89-99.

\bibitem{AlbeverioTorbin2004} Albeverio S.,  Torbin G., Image measures of infinite product measures and generalized Bernoulli convolutions. \emph{Transactions of Dragomanov National Pedagogical University. Series~1: Phys.-Math. Sciences.} \textbf{5}(2004), 248--264.

\bibitem{AKNT1 and 3} Albeverio S.,  Kondratiev Yu., Nikiforov R., Torbin G. On new fractal phenomena connected with infinite linear IFS. \textit{Math. Nachr.}, \textbf{290} (2017), No. 8-9, 1163–1176.     arXiv:1507.05672.


\bibitem{BPT} Baranovskyi O., Pratsiovytyi M., Torbin G. Ostrogradskyi-Sierpinski-Pierce series and their applications. K.: Naukova Dumka, 2013. - 288 p.

\bibitem{Galambos} Galambos J. The ergodic properties of the denominators in the Oppenheim expansion of real numbers into infinite series of rationals. \emph{The Quarterly Journal of Mathematics}, \textbf{21}(1970), No.2, 177-191.


\bibitem{GNT2016}  Garko I., Nikiforov R., Torbin G. On G-isomorphism of probabilistic and dimensional theories of real numbers and fractal faithfulness of systems of coverings. \emph{ Probability theory and Mathematical Statistics}, \textbf{94}(2016), 16-35.



\bibitem{GNT NZ2014(1)} Garko I., Nikiforov R., Torbin G. On G-isomorphism of systems of numerations and faithfulness of systems of coverings.І. \emph{Transactions of the National Pedagogical University (Phys.-Math. Sci.),} \textbf{16}(2014), 120-134.


\bibitem{GNT NZ2014(2)} Garko I., Nikiforov R., Torbin G. On G-isomorphism of systems of numerations and faithfulness of systems of coverings.ІІ.  \emph{Transactions of the National Pedagogical University (Phys.-Math. Sci.)}, \textbf{16}(2014), no. 2, 6-18.


\bibitem{GT2012}  Garko I., Torbin G. On $x-Q_{\infty}$-expansion of real numbers and related problems.//Proceedings of International Conference "Asymptotic methods in the theory of differential equations", Kyiv, National Pedagogical Dragomanov University, 2012, ~48-50.


\bibitem{Hetman} Hetman B.  Metric-topological and fractal theory of representation of real numbers by Engel series: Extended abstract of PhD thesis: 01.01.06. Extended abstract of PhD thesis: 01.01.05., Kyiv, Institute for mathematics, 2012. - 16 p.


\bibitem{INT} Ivanenko G., Nikiforov R., Torbin G. Ergodic approach to investigations of singular probability measures.   \emph{Transactions of the National Pedagogical University (Phys.-Math. Sci.)}, \textbf{7}(2006),   126-142.

\bibitem{Lupain MSTA} Lupain M.  On spectra of probability measures generated by GLS-expansions. \emph{Modern Stochastics: Theory and Applications,} \textbf{3}(3),(2016), 213-221.



\bibitem{Lupain NZ2012} Lupain M. The superpositions of absolutely continuous and singular continuous distribution functions.  \emph{Transactions of the National Pedagogical University (Phys.-Math. Sci.)}, \textbf{3}(2012), No. 2,  128-138.


\bibitem{Lupain NZ2014} Lupain M. Fractal properties of spectra of random variables with independent identically distributed  GLS-symbols. \emph{Transactions of the National Pedagogical University (Phys.-Math. Sci.)}, \textbf{16}(2014), No. 1,  279-295.


\bibitem{Lupain Torbin NZ2014} Lupain M., Torbin G. On new fractal phenomena related to distributions of random variables with  independent  GLS-symbols. \emph{Transactions of the National Pedagogical University (Phys.-Math. Sci.),} \textbf{16}(2014), No. 2,  25-39.



\bibitem{NikTor2012TIMS} ~Nikiforov R., ~Torbin G. Fractal properties of random variables with independent $Q_\infty$-digits, \emph{Theory Probab. Math. Stat.}, \textbf{86} (2013), 169~--~182.


\bibitem{NikTor NZ 2012}  ~Nikiforov R., ~Torbin G. On the Hausdorff-Besicovitch dimension of generalized self-similar sets generated by infinite IFS. \emph{Transactions of the National Pedagogical University (Phys.-Math. Sci.)}, \textbf{13}(2012) No.1,  151-162.



\bibitem{ILT2} Ivanenko G., Lebid M., Torbin G. On the Lebesgue structure and fine fractal properties of some class of infinite Bernoulli convolutions with essential overlaps.  \emph{Transactions of the National Pedagogical University (Phys.-Math. Sci.)}, \textbf{13}(2012) No.2, 47-60.



\bibitem{PSS}   Peres Y., Schlag W. ,  Solomyak B. Sixty years of Bernoulli
convolutions. {\it In Fractal Geometry and Stochastics II}, Progress
in  Probab. vol.46 (Birkh\"{a}user,  2000), 39--65.


\bibitem {PZ} Pratsiovyta I., Zadnipryanyi M.  Silvester expansions for real numbers and their applications. \emph{Transactions of National Pedagogical Dragomanov University: 1. Phys.-Math. Sciences.}  \textbf{10}(2009), 174-189.



\bibitem{ST3} Sinelnyk L. On infinite Bernoulli convolutions generated by binary-lacunary sequences. \emph{Transactions of the National Pedagogical University (Phys.-Math. Sci.)}, \textbf{16}(2012) No.2, 55-60.


\bibitem{ST2} Sinelnyk L., Torbin G., On family of singular probability distributions generated by some subclass of binary-lacunary sequences. \emph{Transactions of the National Pedagogical University (Phys.-Math. Sci.)}, \textbf{16}(2014) No.1, 144-152.


\bibitem{ST1} Sinelnyk L., Torbin G., Asymptotic properties of Fourier-Stiltjes transform of some classes of infinite Bernoulli convolutions. \emph{Transactions of the National Pedagogical University (Phys.-Math. Sci.)}, \textbf{13}(2012) No.2, 229-242.




\bibitem{Solomyak}  B. Solomyak. On the random series $\sum \pm
\lambda ^n$ (an Erd\"{o}s problem). {\it Ann. of Math.} \textbf{142
} (1995), no. 3, 611--625.


\bibitem{TorbinUMJ2005}  Torbin G.  Multifractal analysis of singularly continuous probability
measures. \textit{Ukrainian Math. J.} \textbf{57} (2005), no. 5,
837--857.


\bibitem{Torbin 2002}  Torbin G. Fractal properties of the distributions of random variables with independent Q-symbols. \emph{Transactions of the National Pedagogical University (Phys.-Math. Sci.),} \textbf{3}(2002), 241-252.

\bibitem{Torbin SP 2007}  Torbin G.  Probability distributions with independent Q-symbols and transformations preserving the Hausdorff dimension, {\it Theory of Stochastic Processes},\textbf{13}(2007), 281-293.


\bibitem {TP} Torbin G., Pratsiovyta I., The singularity of the second Ostrogradskyi series, Probab.Th. Math.Stat, \textbf{81}(2010), 187-195.


\bibitem{TP} Torbin G.,  Pratsiovytyi M.  Random variables with independent $Q^*$-symbols. Institute for Mathematics of  NASU:  Random evolutions: theoretical and applied problems, 1992,  95-104.

\bibitem {Zhyhareva} Zhyhareva Yu.  Singular probability distributions related to positive Lueroth series expansion for real numbers. Extended abstract of PhD thesis: 01.01.05., Donetsk, Institute for applied mathematics and mechanics, 2014. - 20 p.


\end{thebibliography}
\end{document}